\documentclass[12pt]{amsart}

\usepackage{amssymb} 
\usepackage{amsthm}
\usepackage{amscd}
\newtheorem{thm}{Theorem}[section]
\newtheorem{lem}[thm]{Lemma}
\newtheorem{lem-dfn}[thm]{Lemma-Definition}
\newtheorem{prop}[thm]{Proposition}
\newtheorem{cor}[thm]{Corollary}

\theoremstyle{definition}
\newtheorem{defn}[thm]{Definition}
\newtheorem{exam}[thm]{Example}
\newtheorem{ex}[thm]{Example}

\newtheorem{quest}[thm]{Question}

\theoremstyle{remark}

\newtheorem{rem}[thm]{Remark}
\numberwithin{equation}{section}

\newcommand{\thmref}[1]{Theorem~\ref{#1}}

\newcommand{\figref}[1]{Figure~\ref{#1}}

\newcommand{\ssref}[1]{Subsection~\ref{#1}}

\DeclareMathOperator{\Spec}{Spec}

\DeclareMathOperator{\Coker}{Coker}

\DeclareMathOperator{\chara}{char}

\DeclareMathOperator{\nr}{nr}
\DeclareMathOperator{\br}{\bar r}

\DeclareMathOperator{\lcm}{lcm}

\newcommand{\m}{\mathfrak m}

\newcommand{\bbZ}{\ensuremath{\mathbb Z}}

\newcommand{\cI}{\mathcal I}

\newcommand{\cO}{\mathcal O}

\newcommand{\ol}[1]{\overline {#1}}

\newcommand{\ce}[1]{\left\lceil #1 \right\rceil}

\newcommand{\ten}{\circle*{0.3}}

\begin{document}
\title{Normal reduction numbers for normal surface singularities 
with application to elliptic singularities of Brieskorn type}
\author{Tomohiro Okuma}
\address[Tomohiro Okuma]{Department of Mathematical Sciences, 
Faculty of Science, Yamagata University,  Yamagata, 990-8560, Japan.}
\email{okuma@sci.kj.yamagata-u.ac.jp}
\author{Kei-ichi Watanabe}
\address[Kei-ichi Watanabe]{Department of Mathematics, College of Humanities and Sciences, 
Nihon University, Setagaya-ku, Tokyo, 156-8550, Japan}
\email{watanabe@math.chs.nihon-u.ac.jp}
\author{Ken-ichi Yoshida}
\address[Ken-ichi Yoshida]{Department of Mathematics, 
College of Humanities and Sciences, 
Nihon University, Setagaya-ku, Tokyo, 156-8550, Japan}
\email{yoshida@math.chs.nihon-u.ac.jp}
\thanks{This work was partially supported by JSPS Grant-in-Aid 
for Scientific Research (C) Grant Numbers, 25400050, 26400053, 17K05216}
\keywords{normal reduction number, geometric genus, 
hypersurface of Brieskorn type}


\begin{abstract}
In this paper, we give a formula for normal reduction number of an integrally closed 
$\m$-primary ideal of a $2$-dimensional normal local ring $(A,\m)$ 
in terms of the geometric genus $p_g(A)$ of $A$. 
Also we compute the normal reduction number of 
the maximal ideal of Brieskorn hypersurfaces. 
As an application, we give a short proof of a classification of Brieskorn 
hypersurfaces having elliptic singularities.  
\end{abstract}

\maketitle

\section{Introduction}
For a given Noetherian local ring $(A,\m)$ and an integrally closed $\m$ primary ideal $I$ with minimal 
reduction $Q$, we are interested in the question;\par
What is the minimal number $r$ such that $I^r \subset Q$ for every $\m$ primary ideal $I$ of $A$ and 
its minimal reduction $Q$ ? 
\par
One example of this direction is the Brian\c con-Skoda Theorem saying;\par

If  $(A,\m)$ is a $d$-dimensional rational singularity (characteristic $0$) or F-rational ring 
(characteristic $p>0$), then $I^d\subset Q$ and $d$ is the minimal possible number in this case
(cf. \cite{LT}, \cite{HH}).\par

We want to ask; what is the minimal number $r$ such that $I^r \subset Q$ for every $\m$ primary
 ideal $I$ of $A$ and its minimal reduction $Q$ ?  
\par
The aim of our paper is to answer this question in the case of normal $2$-dimensional local rings 
using resolution of singularities.
\par
In what follows, we always assume that $(A,\m)$ is an excellent 
two-dimensional normal local domain.  
For any $\m$-primary integrally closed ideal $I \subset A$ 
(e.g. the maximal ideal $\m$) and its minimal reduction $Q$ of $I$, we define two normal reduction numbers as follows:
\begin{eqnarray*}
\nr(I) &=& \min\{n \in \bbZ_{\ge 0} \,|\, \overline{I^{n+1}}=Q\overline{I^n}\}, \\
\br(I) &=& \min\{n \in \bbZ_{\ge 0} \,|\, \overline{I^{N+1}}=Q\overline{I^N} \; \text{for every $N \ge n$}\}.
\end{eqnarray*}

These are analogue of the reduction number $r_Q(I)$ of an ideal $I \subset A$. But in general, $r_Q(I)$ is not 
independent of the choice of a minimal reduction $Q$. 
On the other hand, $\nr(I)=\br(I)$ is \textit{not} known in general. 

Also, we define 
\begin{eqnarray*}
\nr(A)=\max\{\nr(I) \,|\, \text{$I$ is an $\m$-primary integrally closed ideal of $A$}\}, \\
\br(A)=\max\{\br(I) \,|\, \text{$I$ is an $\m$-primary integrally closed ideal of $A$}\}.
\end{eqnarray*}

\par
These invariants of $A$ characterizes \lq\lq  good" singularities.

\begin{ex} [\textrm{See \cite{Li} for (1), \cite{o.h-ell} for (2)}]
Suppose that $A$ is not regular. 
\begin{enumerate}
\item  $A$ is a rational singularity ($p_g(A) =0$) if and only if $\nr(A)= \br(A) =1$.
\item  If $A$ is an elliptic singularity, then $\br(A) =2$, where we say that $A$ is an elliptic 
singularity if the arithmetic genus of the fundamental cycle on any resolution of $A$ is $1$.
\end{enumerate}
\end{ex} 

\par
One of the main aims is to compare these invariants 
with geometric invariants (e.g. geometric genus $p_g(A)$). 
In \cite{OWY1} we have shown that $\nr(A)\le p_g(A)+1$. 
But actually, it turns out that 
we have much better bound (see Theorem \ref{pg-nr}). 

\begin{thm}\label{1.1}
 If $(A,\m)$ is a normal $2$-dimensional local ring, then 
$p_g(A) \ge \binom{\nr(A)}{2}$.
\end{thm}

\par
On the other hand, sometimes we have $\nr(A) = \nr(\m)$.  
For example, if $A=K[[x,y,z]]/(f)$, where $f$ is a homogeneous 
polynomial of degree $d\ge 2$ with isolated singularity, 
it is easy to see $\nr(\m) = d-1$. If $d\le 4$, we can see by 
Theorem \ref{1.1} that $nr(A) = \nr(\m)=d-1$. We do not the answer yet if $d=5$.

\begin{quest}  If $A$ is a homogeneous surface singularity of degree $d$, then 
$\nr(A)= d-1$ ?
\end{quest}
\par
To have examples for this theory, we compute  $\nr(\m)$  of Brieskorn hypersurface 
singularities, that is, two-dimensional normal local domains 
\[
A=K[[x,y,z]]/(x^a+y^b+z^c),
\]
where $K$ is an algebraically closed field of any characteristic and $2 \le a \le b \le c$.  
\par
Note that our approach in this paper will be extended to the case of Brieskorn complete intersection singularity; 
see \cite{MO}.

\par 
We can get explicit value of $\nr(\m)$ in this case. 

\par \vspace{2mm} \par \noindent 
{\bf Theorem \ref{Bries-Hyp}.} 
Let $A$ be a Brieskorn hypersurface singularity as above. 
Put $\m=(x,y,z)A$ and $Q=(y,z)A$. 
Then 
\[
\nr(\m)=\br(\m)=
\left\lfloor \frac{(a-1)b}{a} \right\rfloor. 
\]
Moreover, if we put $n_k=\lfloor \frac{kb}{a}\rfloor$ 
for each $k \ge 0$, then 
\[
\overline{\m^n}=Q^n + xQ^{n-n_1}+x^2Q^{n-n_2}+\cdots 
+x^{a-1}Q^{n-n_{a-1}}.
\]
\par \vspace{2mm}
As an application of the theorem, we can show that 
the Rees algebra $\mathcal{R}(\m)$ is normal if and only if $\br(\m)=a-1$; see Corollary \ref{normalRees}. 
Moreover, we can determine $\ell_A(\m^{n+1}/Q\m^n)$ 
for every $n \ge 0$ and $q(\m)=\ell_A(H^1(X,\mathcal{O}_X(-M))$, where $X \to \Spec A$ denotes 
the resolution of singularity of $\Spec A$ and 
$M$ denotes the maximal ideal cycle on $X$.  

\par \vspace{3mm}
In the last section, we discuss Brieskorn hypersurface with elliptic singularities. 
In fact, the first author proved that if $A$ is an 
elliptic singularity then $\nr(A)=2$. 
In particular, if $A$ is an elliptic singularity 
then $\nr(\m) \le 2$. 
If, in addition, $A$ is a Brieskorn hypersurace singularity $A=K[[x,y,z]]/(x^a+y^b+z^c)$, 
then our theorem shows that 
$\lfloor (a-1)b/a \rfloor \le 2$. 
Using this fact, we can classify all Brieskorn 
hypersurfaces with elliptic singularity. 
See Theorem \ref{elliptic}. 
\par
We are interested to know if $\nr(A)$ characterizes elliptic singularities or not.
Namely, the question is equivalent to say, if $A$ is not rational or elliptic, then 
does there exist $I$ such that $\nr(I)\ge 3$? 
We can find such ideal for all non-elliptic 
Brieskorn hypersurface singularity except $(a,b,c)=(3,4,6)$ or $(3,4,7)$. 

\bigskip
\section{Normal reduction numbers and geometric genus}

Throughout this paper, let $(A,\m)$ be a two-dimensional excellent normal local domain.
In another word, $A$ is a local domain with a resolution of singularities $f : X \to 
\Spec(A)$.  For a coherent $\mathcal{O}_X$-module $\mathcal{F}$, we denote by 
 $h^i(\mathcal{F})$ the length $\ell_A(H^i(\mathcal{F}))$. 

We define the \textit{geometric genus} of $A$ by 
\[
p_g(A)=h^1(\mathcal{O}_X),
\]
which is independent of the choice of resolution of singularities. 
When $p_g(A)=0$, $A$ is called a \textit{rational singularity}. 

\par \vspace{2mm}
Let $I \subset A$ be an $\m$-primary integrally closed 
ideal. 
Then there exists a resolution of singularity  
$X \to \Spec A$ and an anti-nef cycle $Z$ on $X$ so that 
$I\mathcal{O}_X=\mathcal{O}_X(-Z)$ and $I=H^0(\mathcal{O}_X(-Z))$. 
Then we say that $I$ is \textit{represented} by $Z$ on $X$ and 
write $I=I_Z$. 
Then $I_{nZ}=\overline{I^n}$ for every integer $n \ge 1$.  

\par \vspace{2mm}
In what follows, let $A$, $X$, $I=I_Z$ be as above. 

\par 
The authors have studied $p_g$-ideals in \cite{OWY1,OWY2,OWY3}. 
So we first recall the notion of $p_g$-ideals in terms of $q(kI)$.

\begin{defn} \label{qI}
Put $q(I):= h^1(\mathcal{O}_X(-Z))$ and 
$q(nI)= q(\overline{I^n})$ for every 
integer $n \ge 0$.
\end{defn}

\begin{thm}[\cite{OWY1}]
The following statements hold. 
\begin{enumerate}
\item $0 \le q(I) \le p_g(A)$. 
\item $q(kI) \ge q((k+1)I)$ for every integer $k \ge 1$. 
\item $q(nI) = q((n+1)I)=q((n+2)I) = \cdots$ for some integer $n \ge 0$.  
\end{enumerate}
\end{thm}

\begin{defn}[\cite{OWY1}] \label{pg-ideal}
The ideal $I$ is called the \textit{$p_g$-ideal} if 
$q(I)=p_g(A)$. 
\end{defn}

\begin{exam} \label{pg-exam}
Any two-dimensional excellent normal local domain over an algebraically closed field admits a $p_g$-ideal. 
Moreover, if $A$ is a rational singularity,
then every $\m$-primary integrally closed ideal is 
a $p_g$-ideal. 
\end{exam}

\subsection{Upper bound on normal reduction numbers}

Let $Q$ be a minimal reduction of $I$. 
Then there exists a nonnegative integer $r$ such that 
$\overline{I^{r+1}}=Q\overline{I^r}$. 
This is independent of the choice of a minimal reduction 
$Q$ of $I$ (see e.g. \cite[Theorem 4.5]{Hu}). 
So we can define the following notion. 

\begin{defn}[\textbf{Normal reduction number}]
Put 
\begin{eqnarray*}
\nr(I) &=& \min\{n \in \bbZ_{\ge 0} \,|\, \overline{I^{n+1}}=Q\overline{I^n}\},\\
\br(I) &=& \min\{n \in \bbZ_{\ge 0} \,|\, \overline{I^{N+1}}=Q\overline{I^N} \;\text{for every}\; N \ge n \}.
\end{eqnarray*}
We call them the \textit{normal reduction numbers} of $I$.  
We also define 
\begin{eqnarray*}
\nr(A)&=&\max\{\nr(I)\,|\, \text{$I$ is a $\m$-primary integrally closed ideal of $A$}\}, \\
\br(A)&=&\max\{\br(I)\,|\, \text{$I$ is a $\m$-primary integrally closed ideal of $A$}\},
\end{eqnarray*}
which are called 
the \textit{normal reduction numbers} of $A$. 
\end{defn}

\par \vspace{2mm}
Our study on normal reduction numbers is motivated by 
the following observation:
For an $\m$-primary ideal $I$ in a two-dimensional 
excellent normal local domain $A$, $I$ is a $p_g$-ideal 
if and only if $\br(I)=1$. 

\par \vspace{2mm}
By definition, 
$\nr(I) \le \br(I)$ holds in general. 
In the next section, 
we show that $\nr(\m)=\br(\m)$ holds true 
for any Brieskorn hypersurface $A=K[[x,y,z]]/(x^a+y^b+z^c)$. 
But it seems to be open whether equality always holds
for other integrally closed $\m$-primary ideals. 

\begin{quest}
When does $\nr(I)=\br(I)$ hold? 
\end{quest}

\par \vspace{2mm}
In order to state the main result in this section, we recall the following lemma, which gives a relationship 
between $\nr(I)$ and $q(kI)$. 

\begin{lem} \label{qI-nrI}
For any integer $n \ge 1$, we have 
\[
2 \cdot q(nI) + \ell_A(\overline{I^{n+1}}/Q\overline{I^n})
=q((n+1)I)+q((n-1)I).  
\]
\end{lem}

\begin{proof}
Assume $Q=(a,b)$ and consider the exact sequence  
\[
0\to \cO_X((n-1)Z) \to \cO_X(-Z)(-nZ)^{\oplus 2}\to \cO_X(-(n+1)Z)\to 0,
\]
where the map $\cO_X(-nZ)^{\oplus 2}\to \cO_X(-(n+1)Z)$ is defined by  
$(x,y)\mapsto ax+by$ as in Lemma 4.3 of \cite{OWY3}. Taking the cohomology long exact sequence, 
we have the following exact sequence$:$
\[
\begin{array}{lll}
& \to \,H^0(\mathcal{O}_X(-nZ))^{\oplus 2} 
& \stackrel{\varphi}{\to} \,H^0(\mathcal{O}_X(-(n+1)Z)) \\
  \to \,H^1(\mathcal{O}_X(-(n-1)Z)) 
& \to \,H^1(\mathcal{O}_X(-nZ))^{\oplus 2} 
& \to \,H^1(\mathcal{O}_X(-(n+1)Z)) \to 0.
\end{array}
\]
Since $\Coker(\varphi)\cong \overline{I^{n+1}}/Q\overline{I^n}$, we obtain the required assertion. 
\end{proof}

The lemma gives another description of $\nr(I)$ in terms of 
$q(kI)$:
\[
\nr(I)=\min\{n \in \bbZ_{\ge 0} \,|\, 
\text{$q((n-1)I)$, $q(nI)$, $q((n+1)I)$ forms an arithmetic sequence} \}.
\]
In particular, 
\[
\nr(I) \le \min\{n \in \bbZ_{\ge 0}\,|\, q((n-1)I)=q(nI)=q((n+1)I)=\cdots \}=\br(I).
\]

If the following question has an affirmative answer for $I$, then $\nr(I)=\br(I)$ holds true. 

\begin{quest} 
When is $\ell_A(\overline{I^{n+1}}/Q\overline{I^n})$ a non-increasing function of $n$?
\end{quest}

\par \vspace{2mm}
The main result in this section is the following theorem, 
which refines an inequality $\nr(I) \le p_g(A)+1$;
see \cite[Lemma 3.1]{OWY2}. 

\begin{thm} \label{pg-nr}
For any $\m$-primary integrally closed ideal $I \subset A$, we have 
\[
p_g(A) \ge \binom{r}{2}  + q(rI),
\] 
where $r=\nr(I)$. 
In particular, $p_g(A) \ge \binom{\nr(A)}{2}$. 
\end{thm}

\begin{proof}
Suppose $\nr(I)=r$. 
Then since $\overline{I^{k+1}} \ne Q\ \overline{I^k}$ 
for every $k=1,2,\ldots,r-1$ and 
$\overline{I^{r+1}} = Q\ \overline{I^r}$, we have 
\begin{eqnarray*}
q((r-1)I)-q(rI) &=& q(rI)-q((r+1)I), \\
q((r-2)I)-q((r-1)I)& \ge & q((r-1)I)-q(rI)+1, \\
&\vdots & \\
p_g(A)-q(I) &\ge & q(I) -q(2 I)+1.
\end{eqnarray*} 
Thus if we put $a_k=q((r-k)I)$ for $k=0,1,\ldots,r$, then 
we get
\[
a_k-a_{k-1} \ge  a_{k-1}-a_{k-2}+1 \ge \cdots  
\ge  \big\{a_1-a_0 \big\} +(k-1) \ge k-1. 
\]
Hence 
\[
p_g(A)=a_r=\sum_{k=1}^r (a_k-a_{k-1})+a_0
\ge \sum_{k=1}^r (k-1) + a_0 =\dfrac{r(r-1)}{2}+q(rI),
\]
as required.
\par  
The last assertion immediately follows from the 
definition of $\nr(A)$.  
\end{proof}

\par 
The above theorem gives a best possible bound. 
See also the next section. 

\begin{exam} 
If 
$p_g(A) < \binom{\nr(J)+1}{2}$
for some $\m$-primary 
integrally closed ideal $J \subset A$, then 
$\nr(A)=\nr(J)$. 
\end{exam}

\begin{proof}
Suppose $\nr(A) \ne \nr(J)$. 
Then $\nr(A) \ge \nr(J)+1$. 
By assumption and the theorem, we have 
\[
\binom{\nr(A)}{2}  \le p_g(A) <  \binom{\nr(J)+1}{2} 
\le \binom{\nr(A)}{2}.
\]
This is a contradiction. 
Therefore $\nr(A)=\nr(J)$.  
\end{proof}

\medskip
\section{Normal reduction numbers of the maximal ideal of 
Brieskorn hypersurafaces}

Let $K$ be a field of any characteristic, and let 
$a$,$b$,$c$ be integers with $2 \le a \le b \le c$. 
Then a hypersurface singularity
\[
A=K[[x,y,z]]/(x^a+y^b+z^c),\quad 
\m=(x,y,z)A
\]
is called a \textit{Brieskorn hypersurface singularity} if $A$ is normal.  
\subsection{Normal reduction number of the maximal ideal}
The main purpose in this section to give a formula for 
the reduction number of the maximal ideal 
$\m$ in a hypersurface of Brieskorn type$:$ 
$A=K[[x,y,z]]/(x^a+y^b+z^c)$. 
Namely, we prove the following theorem. 

\begin{thm} \label{Bries-Hyp}
Let $A=K[[x,y,z]]/(x^a+y^b+z^c)$ be a Brieskorn hypersurface singularity. 
If we put $Q=(y,z)A$ and 
$n_k=\lfloor \frac{kb}{a} \rfloor$ for 
$k=1,2,\ldots,a-1$, then $\m=\overline{Q}$ and we have 
\begin{enumerate} \setlength{\parskip}{1mm}
\item $\overline{\m^n}=Q^n+xQ^{n-n_1}+x^2Q^{n-n_2}
+ \cdots +x^{a-1}Q^{n-n_{a-1}}$ for every $n \ge 1$. 
\item $\br(\m)=\nr(\m)=n_{a-1}$.
\item $\overline{\mathcal{R}'(\m)}$ and $\overline{G}(\m)$
are Cohen-Macaulay. 
\end{enumerate}
\end{thm}

\begin{rem}
Note $0:=n_0 \le n_1 < n_2 < \cdots < n_{a-1}$. 
In particular, $n_{k} \ge k$ for each $k=0,1,\ldots,a-1$.  
\end{rem}

\par \vspace{3mm}
In the following, we use the notation in this theorem and prove it. 

\begin{lem} \label{socle-k}
For integers $k$, $n$ with $n \ge 1$ and $1 \le n \le a-1$, we have 
that $x^k \in \overline{Q^n}$ if and only if $n \le n_k$.  
\end{lem}  

\begin{proof}
Suppoese $n \le n_k$.  Then 
\[
(x^k)^a=(x^a)^k=(-1)^k (y^b+z^c)^k \in Q^{bk} \subset Q^{an_k} =(Q^{n_k})^a. 
\] 
Hence $x^k \in \overline{Q^{n_k}} \subset \overline{Q^n}$. 
\par \vspace{2mm}
Next, we prove the converse. 
Suppose $x^k \in \overline{Q^n}$. 
Then there exists a nonzero element $c \in A$ such that $c(x^k)^{\ell} \in Q^{n\ell}$ for 
all large integers $\ell$. 
By Artin-Rees' lemma (\cite[Theorem 8.5]{Ma}), we can choose an integer $\ell_0 \ge 1$ such that 
$Q^{\ell} \cap cA =cQ^{\ell-\ell_0}$ for every $\ell \ge \ell_0$. 
\par 
Now suppose that $n \ge n_k+1$. Since $\frac{kb}{a}+\frac{1}{a} \le n_k+1 \le n$, we get 
\[
(y^b+z^c)^{k\ell}=(-1)^kx^{ka\ell} \in Q^{na\ell} \colon c \subset Q^{na\ell -\ell_0} \subset Q^{(n_k+1)a\ell-\ell_0} 
\subset Q^{(bk+1)\ell-n_0}
\]
for sufficiently large $\ell$. 
This implies that $y^{bk\ell} \in (y^{bk\ell+1}, z)$ and this is a contradiction because $y$, $z$ 
forms a regular sequence.  
Therefore $n \le n_k$, as required. 
\end{proof}

\begin{cor} \label{L_n}
For an integer $n \ge 1$, if we put 
\[
L_n=Q^n+xQ^{n-n_1}+x^2Q^{n-n_2}+ \cdots + x^{a-1}Q^{n-n_{a-1}}, 
\]
then $Q^n \subset L_n \subset \overline{Q^n} = \overline{\m^n}$. 
\end{cor}

\begin{proof}
It is enough to prove $x^ky^iz^j \in \overline{Q^n}$ if and only if 
$i+j \ge n-n_k$. 
In fact, since $Q=(y, z)$ is a parameter ideal in $A$,   
\cite[Corollary 6.8.13]{SH} and Lemma \ref{socle-k} 
imply 
\begin{eqnarray*}
x^ky^iz^j \in \overline{Q^n} 
& \Longleftrightarrow & x^ky^{i-1}z^j \in \overline{Q^{n-1}} \\
& \Longleftrightarrow & \cdots\cdots \\
& \Longleftrightarrow & x^k  \in \overline{Q^{n-i-j}} \\
& \Longleftrightarrow & n-n_k \le i+j .
\end{eqnarray*}
Hence $L_n \subset \overline{Q^n}$. 
\end{proof}

Put $d=\gcd(a,b)$, $a'=\frac{a}{d}$ and $b'=\frac{b}{d}$. 
If we put
\[
I_n=(x^ky^iz^j \,|\, kb'+ia'+ja' \ge n)A
\] 
for every $n \ge 1$, then $\{I_n\}_{n=1,2,\cdots}$ is 
a filtration of $A$. 

\begin{lem} \label{G-reduced}
$G(\{I_n\})$ is always reduced. 
In particular, $\mathcal{R}'(\{I_n\})$ is a Gorenstein normal domain. 
\end{lem}

\begin{proof}
One can easily see 
\begin{equation}
G(\{I_n\}) \cong 
\left\{
\begin{array}{ll}
K[X,Y,Z]/(X^a+Y^b+Z^c) & \text{if}\; b=c \\[2mm] 
K[X,Y,Z]/(X^a+Y^b)  &  \text{if}\; b<c.  
\end{array}
\right.
\end{equation}
By assumption, $K[X,Y,Z]/(X^a+Y^b+Z^c)$ is a normal domain. 
If $\chara K=0$, then $K[X,Y,Z]/(X^a+Y^b)$ is reduced. 
Otherwise, we put $p=\chara K>0$. 
Since $A$ is normal, we have that 
$p$ does \textit{not} divide $\gcd(a,b)=d$. 
Hence $K[X,Y]/(X^a+Y^b)$ is reduced. 
\par 
As $A$ is normal,  $R=\mathcal{R}'(\{I_n\})$ is a Gorenstein normal domain because $G(\{I_n\}) \cong R/t^{-1}R$. 
\end{proof}

\begin{lem} \label{a-Veronese}
$L_{n}=I_{na'}$ for every $n \ge 1$. 
\end{lem}

\begin{proof}
Since $L_n$ and $I_{na'}$ are monomial ideals, 
it suffices to show that 
$x^ky^iz^j \in  L_n$ if and only if $x^ky^iz^j \in  I_{na'}$. 
But this is clear from the definition. 
\end{proof}

We are now ready to prove the theorem. 
\begin{proof}[Proof of Theorem $\ref{Bries-Hyp}$]
(1)  Since $\mathcal{R}'(\{I_n\})$ is normal 
by Lemma \ref{G-reduced}, 
we have that every $I_n$ is integrally closed. 
In particular, $L_n = I_{na'}$ is also integrally closed 
by Lemma \ref{a-Veronese}.  
Therefore $L_n=\overline{Q^n}=\overline{\m^n}$ by Corollary \ref{L_n}. 
\par \vspace{1mm}\par \noindent
(2) One can easily see that 
$L_{n+1}=QL_n$ if and only if $n \ge n_{a-1}$.
Hence (2) is immeadiately follows from (1).  
\par \vspace{1mm} \par \noindent
(3)  $\overline{\mathcal{R}'(\m)}$ is Cohen-Macaulay since it is 
a Veronese subring of a Cohen-Macaulay 
ring $\mathcal{R}'(\{I_n\})$. 
Then $\overline{G}(\m)=\overline{\mathcal{R}'(\m)}/t^{-1}
\overline{\mathcal{R}'(\m)}$ is also Cohen-Macaulay by \cite[Theorem 4.1]{OWY2}.
\end{proof}

\begin{cor} \label{normalRees}
Let $(A,\m)$ be a Brieskorn hypersurface as in Theorem $\ref{Bries-Hyp}$. 
Then 
\begin{enumerate}
\item $\mathcal{R}(\m)$ is normal if and only if $\br(\m) = a-1$. 
\item $\overline{\mathcal{R}(\m)}$ is Cohen-Macaulay if  and only if $\br(\m) = 1$. 
\item $\m$ is a $p_g$-ideal if and only if $a=2$ and $\br(\m)=1$.  
\end{enumerate}
\end{cor}

\begin{proof}
(1)  Suppose $\br(\m)=a-1$. 
Then $n_{a-1}=a-1$ by (1) and this implies that $n_k=k$ for each $k=1,2,\ldots,a-1$. 
Then one can easily see that $\overline{\m^n} = (Q,x)^n=\m^n$ for 
every $n \ge 1$. Hence $\mathcal{R}(\m)$ is normal. 
\par 
Conversely, if $\mathcal{R}(\m)$ is normal, then 
$\overline{\m^n} = \m^n=(Q,x)^n$. Then $n_{a-1}=a-1$. 
\par \vspace{1mm} \par \noindent
(2) Since $F=\{\overline{\m^n}\}$ is a good $\m$-adic filtration, 
$\overline{\mathcal{R}(\m)}=\mathcal{R}(F)$ 
is Cohen-Macaulay if and only if $G(F)$ is Cohen-Macaulay and $\br(\m)-2=a(G(F))<0$ by  
\cite[Part 2, Corollary 1.2]{Goto-Nishida} and \cite[Theorem 3.8]{Hoa-Zarzuela}. 
\par \vspace{1mm} \par \noindent
(3) $\m$ is a $p_g$-ideal if and only if 
$R(\m)$ is normal and Cohen-Macaulay. 
Hence the assertion follows from (1),(2).  
\end{proof}

\subsection{$q(\m)$ and $\ell_A(\overline{\m^{n+1}}/Q\overline{\m^n})$}

In the proof of Theorem \ref{Bries-Hyp}, we gave a formula of the integral closure of $\m^n$. 
As an application, we give a formula of $q(\m)$ for Brieskorn hypersurface singularities. 

\begin{prop} \label{qM}
Let $A=K[[x,y,z]]/(x^a+y^b+z^c)$ be a Brieskorn hypersurface singularity. 
Under the same notation as in Thoerem \ref{Bries-Hyp}, we have 
\begin{enumerate}
\item 
$\ell_A(\overline{\m^{n+1}}/Q\overline{\m^n}) 
= \max\left(a-\lceil\frac{a(n+1)}{b}\rceil,\;0\right)$.
\item $q(\m)=p_g(A)-\displaystyle{\sum_{k=1}^{a-1}} (n_k-n_{k-1})(a-k)$. 
\end{enumerate}
\end{prop}

\begin{proof}
\par 
Suppose $n_k \le n < n_{k+1}$ for some $0 \le k \le a-2$. 
Then 
$\overline{\m^n}=Q^n+xQ^{n-n_1} +\cdots 
+ x^kQ^{n-n_k}+(x^{k+1})$ 
and $x^{k+1},x^{k+2},\ldots,x^{a-1}$ 
forms a $K$-basis of 
$\overline{\m^{n+1}}/Q\overline{\m^n}$ and thus $\ell_A(\overline{\m^{n+1}}/Q\overline{\m^n})=a-1-k$. 
Hence 
\[
\ell_A(\overline{\m^{n+1}}/Q\overline{\m^n}) 
= \left\{
\begin{array}{ll}
a-1-k & \text{if}\; n_k \le n < n_{k+1},\; k=0,1,\ldots,a-2;  \\[1mm]
0 & \text{if}\; n \ge n_{a-1}. 
\end{array}
\right.
\]
Moreover, one can easily see 
$k=a-\lceil\frac{a(n+1)}{b}\rceil-1$. 

\par \vspace{2mm} \par \noindent
(2) Put $a_n=p_g(A)-q(n\m)$ and $v_n=\ell_A(\overline{\m^{n+1}}/Q\overline{\m^n})$ for every $n \ge 0$. 
Then $a_0=0$ and $\{a_n\}$ is an increasing sequence and $a_{n+1}=a_n$ for sufficiently large $n$. 
By Lemma \ref{qI-nrI}, we have 
\[
0=a_{n+1}-a_n=a_n-a_{n-1}-v_n = \cdots = a_1 - a_0 - \text{\small $\displaystyle{\sum_{k=1}^n} v_k$} 
\]
for sufficiently large $n \ge 1$. 
Hence (1) yields 
\[
p_g(A)-q(\m)=a_1= \sum_{k=1}^n v_k =\sum_{k=1}^{a-1} (n_k-n_{k-1})(a-k),
\]
as required. 
\end{proof}

When $a=2$, one can obtain the following. 

\begin{exam}
Let $A=K[[x,y,z]]/(x^2+y^b+z^c)$ be a Brieskorn hypersurface singularity and 
put $r=\lfloor \frac{b}{2} \rfloor$. Then 
\begin{enumerate}
\item $q(i \m)=\left\{
\begin{array}{ll}
p_g(A)-i(r-1) + \binom{i}{2} & \text{if}\;1 \le i \le r-1; \\[1mm]
p_g(A)-\binom{r}{2} & \text{if}\; i \ge r. 
\end{array}
\right.$ \\
\item The normal Hilbert coefficients of $\m$ are given as follows:
\[
\overline{e}_0(\m)=2, \quad \overline{e}_1(\m)=r, \quad \overline{e}_2(\m)=\binom{r}{2},
\]
where
\[
\ell_A(A/\overline{I^{n+1}}) = \overline{e}_0(I) \binom{n+2}{2} - \overline{e}_1(I)\binom{n+1}{1}+\overline{e}_2(I)
\]
for sufficiently large $n$. 
\end{enumerate}
\end{exam}
\subsection{Geometric genus}

In this subsection, let us consider a graded ring 
\[
B=K[x,y,z]/(x^a+y^b+z^c)
\]
with  $\deg x=q_0=bc$, $\deg y=q_1=ac$ and 
$\deg z=q_2=ab$. 
Put $\m=(x,y,z)A$ and $D=abc$. 
In particular,  the $a$-invariant of $B$ is given by 
$a(B)=D-q_0-q_1-q_2$. 
Also we have that $A=\widehat{B_{\m}}$ is the completion of the local ring $B_{\m}$. 
Then we can calculate $p_g(A)$ using this formula. 

\begin{lem} \label{pg-formula}
Under the above notation, we have 
\[
p_g(A)=\sum_{i=0}^{a(B)} \dim_{K} B_i =\sharp\{(t_0, t_1,t_2) \in \bbZ_{\ge 0}^{\oplus 3} \,|\, D - q_0-q_1-q_2 \ge q_0t_0+q_1t_1+q_2t_2\}
\]
\end{lem}

We can find many examples of Brieskorn hypersurface with 
$p_g(A)=p$ for a given $p \ge 1$  if $\nr(\m)=1,2$. 

\begin{exam}
Let $p \ge 1$ be an integer. 
\begin{enumerate}
\item If $A=\mathbb{C}[[x,y,z]]/(x^2+y^3+z^{6p+1})$, then $p_g(A)=p$ and $\nr(\m)=\br(\m)=1$.  
\item If $A=\mathbb{C}[[x,y,z]]/(x^2+y^4+z^{4p+1})$, then $p_g(A)=p$ and $\nr(\m)=\br(\m)=2$.  
\end{enumerate}
\end{exam}

\begin{exam}
Let $k \ge 1$ be an integer. 
\begin{enumerate}
\item Put $A=\mathbb{C}[[x,y,z]]/(x^2+y^6+z^{10k+i})$ for $i=0,1,\ldots,9$. 
Then $\nr(\m)=\br(\m)=3$ and 
\[
p_g(A)=\left\{
\begin{array}{ll}
6k, & \text{(if $i=0,1,2$)}; \\
6k+1, & \text{(if $i=3,4,5$)}; 
\end{array}
\right.
\qquad 
p_g(A)=\left\{
\begin{array}{ll}
6k+3, & \text{(if $i=6,7,8$)}; \\
6k+4, & \text{(if $i=9,10,11$)}. 
\end{array}
\right.
\]
\item Put $A=\mathbb{C}[[x,y,z]]/(x^2+y^7+z^{14k+i})$ for $i=0,1,\ldots,13$. 
Then $\nr(\m)=\br(\m)=3$ and 
\[
p_g(A)=\left\{
\begin{array}{ll}
9k, & \text{(if $i=0,1,2$)}; \\
9k+1, & \text{(if $i=3,4$)}; \\
9k+2, & \text{(if $i=5$)}; \\
9k+3, & \text{(if $i=6,7,8$)}. 
\end{array}
\right.
\qquad 
p_g(A)=\left\{
\begin{array}{ll}
9k+4, & \text{(if $i=9$)}; \\
9k+5, & \text{(if $i=10,11$)}; \\
9k+6, & \text{(if $i=12,13$)}.
\end{array}
\right.
\]
\end{enumerate}
\end{exam}

We discuss when $p_g(A)=\binom{\nr(\m)}{2}$ holds. 

\begin{prop} \label{boundary-ex}
Let $A=\mathbb{C}[[x,y,z]]/(x^a+y^b+z^c)$ with $2 \le a \le b \le c$. 
Then $p_g(A)=\binom{\nr(\m)}{2}$ if and only if 
one of the following cases: 
\begin{enumerate}
\item[$\bullet$] $(a,b,c)=(2,2,n)$ $(n \ge 1)$.   In this case,   $\nr(A)=\nr(\m)=1$ and $p_g(A)=0$. 
\item[$\bullet$] $(a,b,c)=(2,3,3)$, $(2,3,4)$, $(2,3,5)$.   In this case,   $\nr(A)=\nr(\m)=1$ and $p_g(A)=0$.  
\item[$\bullet$] $(a,b,c)=(2,4,4),(2,4,5),(2,4,6),(2,4,7)$. 
  In this case,    $\nr(A)=\nr(\m)=2$ and $p_g(A)=1$. 
\item[$\bullet$]  $(a,b,c)=(2,2r,2r),(2,2r,2r+1),(2,2r,2r+2) $ $(r \ge 3)$.    In this case,  
$\nr(A)=\nr(\m)=r$ and $p_g(A)=\binom{r}{2} \ge 3$. 
\item[$\bullet$]  $(a,b,c)=(2,2r+1,2r+1),(2,2r+1,2r+2)$ $(r \ge 2)$.  
  In this case,  $\nr(A)=\nr(\m)=r$ and $p_g(A)=\binom{r}{2}$. 
\item[$\bullet$]  $(a,b,c)=(3,3,3),(3,3,4),(3,3,5)$.  In this case,  $\nr(A)=\nr(\m)=2$ and $p_g(A)=1$. 
\item[$\bullet$]  $(a,b,c)=(3,3s+1,3s+1)$.  In this case,   
 $\nr(A)=\nr(\m)=2s$ and $p_g(A)=\binom{2s}{2}$. 
 \item[$\bullet$]  $(a,b,c)=(3,3s+2,3s+2),(3,3s+2,3s+3)$.  In this case,  $\nr(A)=\nr(\m)=2s+1$ and $p_g(A)=\binom{2s+1}{2}$. 
\end{enumerate}
\end{prop} 

\begin{proof}
We give a the proof  of only if part. 
Put $r=\nr(\m)$. 
By  Theorem \ref{Bries-Hyp}, we have $\lfloor \frac{(a-1)b}{a}\rfloor$. 
So we can write $(a-1)b=ra+\varepsilon$, where $\varepsilon$ is 
an integer with $0 \le \varepsilon \le a-1$. 
Now suppose 
\[
abc-bc-ca-ab \ge bc \lambda_0 + ca \lambda_1 + ab \lambda_2. 
\]
Then 
\begin{equation} \label{range}
(ra+\varepsilon)c-ca-ab \ge bc \lambda_0 + ca \lambda_1 + ab \lambda_2. \tag{eq.pg} 
\end{equation}
Suppose $\lambda_0=0$. Then 
\[
\left(r-1+\frac{\varepsilon}{~a~}-\lambda_1 \right)\frac{~c~}{b} 
\ge \lambda_2+1 
\]
and thus $\lambda_1 < r-1+\frac{\varepsilon}{~a~}$. 
By Lemma \ref{pg-formula} and assumption, we have 
\begin{eqnarray*}
\binom{r}{2} = p_g(A) 
&\ge & \sum_{k=0}^{r-1} 
\left\lfloor\left(r-1+\frac{\varepsilon}{~a~}-k \right)\frac{~c~}{b} \right\rfloor \\
&\ge & \sum_{k=0}^{r-2} 
\left\lfloor\left(r-1-k \right)\frac{~c~}{b} \right\rfloor + 
\left\lfloor\frac{\varepsilon}{~a~} \cdot \frac{~c~}{b} \right\rfloor \\
&\ge & \sum_{k=0}^{r-2} 
\left(r-1-k \right) + 
\left\lfloor\frac{\varepsilon}{~a~} \cdot \frac{~c~}{b} \right\rfloor \ge \binom{r}{2}. 
\end{eqnarray*}
Hence $\left(r-1+\frac{\varepsilon}{~a~}-k \right)\frac{~c~}{b} < r-k$ 
for each $k=0,1,\ldots,r-1$. 
Moreover, if $\lambda_0 \ge 1$, then since $\lambda_1=\lambda_2=0$ 
does not satisfy the condition (\ref{range}) by Theorem \ref{pg-nr}, we get 
\[
abc-bc-ca-ab < bc, \quad \text{that is,}\quad 
\frac{2}{~a~} + \frac{~1~}{b}+\frac{~1~}{c} > 1. 
\]
This implies $a=2,3$. 
If $a=2$, then $\varepsilon =0,1$. If $a=3$, then 
$\varepsilon = 0,1,2$. 
\par
Now suppose $a=3$ and $\varepsilon =2$. 
Then as $2b=3r+2$, we can write $r=2s$, $b=3s+1$, where 
$s \ge 1$. 
Moreover, the condition holds true if and only if 
$(2s+\frac{2}{3}-1)\frac{c}{3s+1}< 2s$. 
This means $c < 3s+1 + \frac{3s+1}{6s-1}$. 
Hence $c=3s+1$ because $c \ge b=3s+1$. 
Similarly, easy calculation yields the required assertion. 
\end{proof}

\subsection{Weighted dual graph}\label{ss:WDG}
In this subsection, let us explain how to construct 
the weighted dual graph of the minimal good resolution of singularity $X \to \Spec A$ for 
a Brieskorn hypersurface singularity $A=K[[x,y,z]]/(x^a+y^b+z^c)$. 
Though it is obtained in \cite{K-N}, we use the notation of \cite{MO} which studies complete intersection singularities of Brieskorn type.
Let $E$ be the exceptional set of $X \to \Spec A$ and 
$E_0$ the central curve with genus $g$ and $E_0^2=-c_0$. 
We define positive integer $a_i$,$\ell_i$,$\alpha_i$, $\lambda_i$, $\widehat{g_i}$ (i=1,2,3), $\widehat{g}$ and $\ell$ as follows:
\[
\begin{array}{lll}
a_1 =a, & a_2=b, & a_3=c, \\[1mm]
\ell_1 = \lcm(b,c), & \ell_2 = \lcm(a,c), & 
\ell_3 = \lcm(a,b), \\[1mm]
\alpha_1 = \frac{a_1}{(a_1,\ell_1)}, & 
\alpha_2 = \frac{a_2}{(a_2,\ell_2)}, &
\alpha_3 = \frac{a_3}{(a_3,\ell_3)}, \\[2mm]
\lambda_1 = \frac{\ell_1}{(a_1,\ell_1)}, &
\lambda_2 = \frac{\ell_2}{(a_2,\ell_2)}, &
\lambda_3 = \frac{\ell_3}{(a_3,\ell_3)}, \\[2mm]
\hat g_1=(b,c), & \hat g_2=(a,c), & 
\hat g_3=(a,b).  
\end{array}
\]
We put $\widehat{g}=\frac{abc}{\lcm(a,b,c)}$ and 
$\ell=\lcm(a,b,c)$, and define integers $\beta_i$ 
by the following condition:
\[
\lambda_i\beta_i+1 \equiv 0 \pmod{\alpha_i},\quad 
0 \le \beta_i < \alpha_i. 
\]
Then $E_0$ has 
$\hat g_1+\hat g_2 +\hat g_3$ branches. 
For each $w=1,2,3$, we have $\hat g_w$ branches
\[
B_w : \quad E_{w,1} - E_{w,2} - \cdots - E_{w,s_w},
\]
where $E_{w,j}^2=-c_{w,j}$ and 
\[
\frac{\alpha_w}{\beta_w} = [[c_{w,1},c_{w,2},\ldots,c_{w,s_w}]] 
\] 
is a Hirzebruch-Jung continued fraction if $\alpha_w \ge 2$; we regard $B_w$ empty if $\alpha_w=1$. 
Moreover, we have 
\[
2g-2=\hat g- \sum_{w=1}^{3} 
\hat g_w, \qquad 
c_0= \sum_{w=1}^{3} \frac{\hat g_w \beta_w}{\alpha_w}
+ \dfrac{\hat g}{\ell}. 
\]
\par
For instance, if $(a,b,c)=(3,4,7)$, then 
\[
\begin{array}{llllll}
a_1 =3, &  a_2=4, & a_3=7, &
\ell_1 = 28, & \ell_2 = 21, & 
\ell_3 = 12, \\[1mm]
\alpha_1 = 3, & 
\alpha_2 = 4, &
\alpha_3 = 7, &
\lambda_1 = 28, &
\lambda_2 = 21, &
\lambda_3 = 12, \\[2mm]
\beta_1 = 2, & 
\beta_2 = 3, &
\beta_3 = 4, &
\hat g_1=1, & \hat g_2=1, & 
\hat g_3=1, \\[2mm]
\hat g=1, & \ell=84. & & & &  
\end{array}
\]
Thus $g=0$ and $c_0=2$. 
Therefore, each irreducible component of $E$ is a rational curve, and the weighted dual graph of $E$ is represented as in \figref{f:347}, where the vertex  \rule{2mm}{2mm} has weight $-4$ and other vertices
$\bullet$ have  weight $-2$.
\vspace{3mm}
\begin{figure}[htbp] \begin{center}
\setlength{\unitlength}{0.5cm}
    \begin{picture}(11,3)(0,-1)
\put(0,0){\makebox(0,0){\rule{2mm}{2mm}}}
\multiput(2,0)(2,0){5}{\ten}
\multiput(6,2)(2,0){2}{\ten}
\put(0,0){\line(1,0){10}}
\put(4,0){\line(1,1){2}}
\put(6,2){\line(1,0){2}}
\put(0,0){\makebox(0,1.5){$B_3$}}
\put(10,0){\makebox(0,1.5){$B_2$}}
\put(8,2){\makebox(0,1.5){$B_1$}}
\end{picture}
\caption{\label{f:347}}
  \end{center}
\end{figure}

See \cite[4.4]{MO} for more details.


\section{Brieskorn hypersurfaces with elliptic singularities}

We use the notation of \ssref{ss:WDG}.
Let $Z_E$ denote the fundamental cycle.

We call $p_f(A):=p_a(Z_E)$ the {\em fundamental genus} of $A$.
The singularity $A$ is said to be {\em elliptic} if $p_f(A)=1$.
We have the following.

\begin{thm}
[{\cite{o.h-ell}}]
\label{t:ellnr}
If $p_f(A)=1$, then $\br(A)=2$.
\end{thm}
\begin{rem}
By \cite{CharRat}, $\nr(A)=1$ if and only if $A$ is rational.
Therefore, $\nr(A)=\br(A)$ if $\br(A)=2$.
\end{rem}


It is natural to ask whether the converse of \thmref{t:ellnr} holds or not.
In the following, we classify Brieskorn hypersurface singularities with $p_f(A)=1$ or $\br(A)=2$ as an application of results in Section 3.
Before doing that, we need the following formula 
of $p_f(A)$ in the case of Brieskorn hypersurfaces. 
Put $\alpha=\alpha_1\alpha_2\alpha_3$. 

\begin{lem}[\textrm{\cite{tomaru-pfBH}, \cite[Theorem 1.7]{K-N}, 
\cite[5.4]{MO}}]\label{t:pf}
If $\lambda_3 \le \alpha$, then $-Z_E^2
=\hat g_{3}\ce{\lambda_3/\alpha _{3}}$ and 
\begin{align*}
p_f(A)&=\frac{1}{2}\lambda_3 \left\{\hat g
-\frac{(2\ce{\lambda_3/\alpha_{3}}-1)\hat g_{3}}{\lambda_3}
-\frac{\hat g_1}{\alpha_1}-\frac{\hat g_2}{\alpha_2}\right\}+1 \\
&=\frac{1}{2}\left(ab-a-b-(2\ce{\lambda_3/\alpha_{3}}-1)(a,b)\right)+1.
\end{align*}
%
\end{lem}

\par
We are now ready to state our result in the case of 
$p_f(A)=1$. 

\begin{thm} \label{elliptic}
$(A,\m)$ is elliptic $($i.e. $p_f(A)=1$$)$ if and only if  $(a, b, c)$ is one of the following.
\begin{enumerate}
\item $(2,3,c)$, $c\ge 6$.
\item $(2,4,c)$, $c\ge 4$.
\item $(2,5,c)$, $5\le c \le 9$.
\item $(3,3,c)$, $c\ge 3$.
\item $(3,4,c)$, $4\le c \le 5$.
\end{enumerate}
\end{thm}
\begin{proof}
If $A$ is elliptic, then by \thmref{t:ellnr} and \thmref{Bries-Hyp},
we have $\lfloor \frac{(a-1)b}{a} \rfloor=2$.
Thus possible pairs $(a, b)$ are:
\[
(2,3), (2,4), (2,5), (3,3), (3,4).
\]
We know that $A$ is rational if $(a,b,c)=(2,3,c)$ with $3\le c \le 5$.
We obtain the assertion by the Lemma \ref{t:pf}; 
for example, $p_f(A)=3-\ce{10/c}$ for $(a,b,c)=(2,5,c)$, and  $p_f(A)=4-\ce{12/c}$ for $(a,b,c)=(3,4,c)$.
\end{proof}

\par
We can classify Brieskorn hypersurface singularities 
with $\br(A)=2$ except $(a,b,c)=(3,4,6)$ or $(3,4,7)$.

\begin{prop}\label{p:br=2}
$\br(A)=2$ if and only if $p_f(A)=1$, except $(a,b,c)=(3,4,6)$, or $(3,4,7)$.
\end{prop}
\begin{proof}
It suffices to check whether $\nr(A)\ge 3$ for singularities with $\br(\m)= 2$ and $p_f(A)\ge 2$.
\par
Suppose $(a,b,c)=(2,5,c)$, $c\ge 10$. Let $Q=(y,z^2)$ and $J=\ol Q$. Then $xz\not\in Q$ and $(xz)^2=(y^5+z^c)z^2\in Q^6=(Q^3)^2$. Hence $\nr(J)\ge 3$.
\par 
Next suppose that $(a,b,c)=(3,4,c)$, $c\ge 8$.
Let $Q=(y,z^2)$ and $J=\ol Q$, again. Then $x^2z\not\in Q$ and $(x^2z)^3=(y^4+z^c)^2z^3\in Q^9=(Q^3)^3$. Hence $\nr(J)\ge 3$.
\end{proof}

Applying the result of \cite{MO}, we can show that the formula for $\br(\m)$ for Brieskorn  complete intersection singularities.
Thus the statement above can be extended to those singularities.

\begin{rem}
Supose that $p_g(A)=3$.
It follows from \thmref{pg-nr} and its proof that $\nr(I)=3$ if and only if $q(I)=1$ and $q(nI)=0$ for $n\ge 2$.
In particular, $q(nI)=q(I)$ for $n\ge 2$ if $q(I)\ge 2$.
\end{rem}

\begin{rem}
If $(a,b,c)=(3,4,6)$ or $(3,4,7)$,  we have the following.
\begin{enumerate}
\item $p_g(A)=3$, $p_f(A)=2$, $h^1(\cO_X(-Z_E))=1$.
\item There exists a point $p\in E$ such that $\m\cO_X=\cI_p\cO_X(-Z_E)$, where $\cI_p\subset \cO_X$ is the ideal sheaf of the point $p$; so $\m=H^0(\cO_X(-Z_E))$, but $\m$ is not represented by $Z_E$.
Note that $H^0(\cO_X(-nZ_E))\ne \ol{\m^n}$.
On the other hand, $\cO_X(-2Z_E)=\cO_X(K_X)$ is generated by global sections. 
By the vanishing theorem, $h^1(\cO_X(-nZ_E))=0$  for $n\ge 2$.
\end{enumerate}
\end{rem}


\providecommand{\bysame}{\leavevmode\hbox to3em{\hrulefill}\thinspace}
\providecommand{\MR}{\relax\ifhmode\unskip\space\fi MR }
\providecommand{\MRhref}[2]{%
  \href{http://www.ams.org/mathscinet-getitem?mr=#1}{#2}
}
\providecommand{\href}[2]{#2}

\end{document}